\documentclass[12pt]{article}
\usepackage{hyperref,amsmath,amsthm,amsfonts,amssymb,amscd,enumerate,tikz}

\usetikzlibrary{positioning}

\theoremstyle{plain}
\newtheorem{theorem}{Theorem}[section]
\newtheorem{proposition}[theorem]{Proposition}

\newtheorem{corollary}[theorem]{Corollary}
\newtheorem{lemma}[theorem]{Lemma}

\newtheorem*{Theorem}{Theorem}

\DeclareMathOperator{\piprod}{\raisebox{-0.1em}{\huge{$\pi$}}\kern -0.2em}

\theoremstyle{definition}

\newtheorem{remark}[theorem]{Remark}
\newtheorem{remarks}[theorem]{Remarks}

\newtheorem*{example*}{Example}
\newtheorem*{Remark}{Remark}

\newcommand{\cg}{\mathcal {G}}

\newcommand{\ci}{\mathcal {I}}

\newcommand{\cp}{\mathcal {P}}
\newcommand{\car}{\mathcal {R}}

\newcommand{\cz}{\mathcal {Z}}

\newcommand{\nn}{{\mathbb N}}

\newcommand{\qq}{{\mathbb Q}}
\newcommand{\rr}{{\mathbb R}}

\newcommand{\zz}{{\mathbb Z}}

\newcommand{\ltwob}{L^2b}

\newcommand{\minus}{{-1}}

\newcommand{\dtwo}{\lfloor d/2 \rfloor}

\newcommand{\n}{{[n]}}
\newcommand{\ah}{\hat{h}}

\newcommand{\ua}{\underline{A}}
\newcommand{\ub}{\underline{B}}
\newcommand{\ug}{\underline{\gG}}
\newcommand{\uab}{(\ua,\ub)}
\newcommand{\zx}{\cz_X}
\newcommand{\zxg}{\cz_{X(G)}}

\newcommand{\tzxg}{\tilde{\cz}_{X(G)}}

\newcommand{\bF}{{\mathbf F}}
\newcommand{\ftwo}{{\mathbf F _2}}

\newcommand{\be}{{\mathbf e}}

\newcommand{\bq}{{\mathbf q}}

\newcommand{\bt}{{\mathbf t}}
\newcommand{\bone}{{\mathbf 1}}

\newcommand{\wt}{\widetilde}

\newcommand{\wH}{\widetilde {H}}
\newcommand{\wh}{\widehat}

\newcommand{\ul}{\underline}

\newcommand{\rh}{{\wt{H}\,}}

\def\clap#1{\hbox to 0pt{\hss#1\hss}}

\newcommand{\comment}[1]{}

\newcommand{\geps}{\varepsilon}
\newcommand{\gf}{\varphi}

\newcommand{\gr}{\rho}
\newcommand{\gs}{\sigma}

\newcommand{\gt}{\tau}

\newcommand{\gl}{\lambda}
\newcommand{\go}{\omega}

\newcommand{\gG}{\Gamma}
\newcommand{\gD}{\Delta}

\newcommand{\edge}{\operatorname{Edge}}
\newcommand{\Ends}{\operatorname{Ends}}

\newcommand{\Ext}{\operatorname{Ext}}

\newcommand{\Ker}{\operatorname{Ker}}

\newcommand{\Lk}{\operatorname{Lk}}

\newcommand{\cone}{\operatorname{Cone}}

\newcommand{\GR}{\operatorname{Gr}}
\newcommand{\cd}{\operatorname{cd}}
\newcommand{\vcd}{\operatorname{vcd}}

\newcommand{\grh}{\GR{H}}
\newcommand{\aas}{\mathrm{w.h.p.}}
\newcommand{\eaas}{\mathrm{w.h.p}}

\newenvironment{enumerate1}{
\begin{enumerate}[\upshape (1)]}%
	{
\end{enumerate}
}
\newenvironment{enumerate1*}{
\begin{enumerate}[\upshape (*1)]}%
	{
\end{enumerate}
}
	{
\end{enumerate}
}

\newenvironment{enumeratea}{
\begin{enumerate}[\upshape (a)]}{
\end{enumerate}
}
\newenvironment{enumeratea'}{
\begin{enumerate}[\upshape (a)$'$]}{
\end{enumerate}
}

\newcommand{\expect}{\mathbb{E}}
\newcommand{\prob}{\mathbf{Pr}}

\newtheorem{theo+}{Theorem}
\newtheorem{prop+}{Proposition}
\newtheorem{coro+}{Corollary}
\newtheorem{lemm+}{Lemma}
\theoremstyle{definition}
\newtheorem{defi+}{Definition}

\theoremstyle{remark}
\newtheorem{rema+}{Remark}

\DeclareMathOperator{\I}{\bf I}

 \numberwithin{equation}{section} 

\usepackage{color}
\usepackage[outerbars,color]{changebar}
\usepackage{ifthen}

\newboolean{usecolor}
\setboolean{usecolor}{true}

\ifthenelse{\boolean{usecolor}}
{
  \definecolor{colore}{cmyk}{0,1,0.6,0}
  \definecolor{coloregen}{cmyk}{0.7,0,1,0}
  \definecolor{coloresimo}{cmyk}{1,0.6,0,0}
}
{
  \definecolor{colore}{cmyk}{0,0,0,1}
  \definecolor{coloregen}{cmyk}{0,0,0,1}
  \definecolor{coloresimo}{cmyk}{0,0,0,1}
}

\begin{document}

\title{Random graph products of finite groups are rational duality groups }
\author{Michael W. Davis\thanks{Partially supported by NSF grant DMS 1007068 and the Institute for Advanced Study.} \and{Matthew Kahle}\thanks{Partially supported by the Institute for Advanced Study.}}
\date{\today} \maketitle
\begin{abstract} Given an edge-independent random graph $G(n,p)$, we determine various facts about the cohomology of graph products of groups for the graph $G(n,p)$.  In particular, the random graph product of a sequence of finite groups is a rational duality group with probability tending to $1$ as $n \to \infty$.  This includes random right angled Coxeter groups as a special case. 
	\smallskip

	\paragraph{AMS classification numbers.} Primary: 05C10, 05C80,  20F36, 20F55, 20F65, 20J06 \\
	Secondary: 57M07. 
	
	\smallskip

\end{abstract}

\section*{Introduction}\addcontentsline{toc}{section}{Introduction}

A simplicial graph $G$ determines a simplicial complex $X(G)$, called its \emph{flag complex} (or its ``clique complex").  The simplices of $X(G)$ are the complete subgraphs of $G$.  

Given a sequence $\ug = (\gG_i)_{i\in \nn}$ of discrete groups indexed by the natural numbers and a graph $G$ with vertex set $[n]$ (where $[n]:=\{1,\dots , n\}$), we construct a new group $\cg$ ($=\cg(G,\ug)$), called the \emph{graph product},  by taking the  free product of the  $\gG_i$, $i \in [n]$, and then imposing the relations that elements in $\gG_i$ commute with elements of  $\gG_j$ whenever $\{i,j\}\in \edge(G)$.  We are mainly interested in the case where $\ug$ is the constant sequence $\gG_i=\gG$, for some group $\gG$.

It turns out that the cohomology of $\cg$ with coefficients in the group ring, $\zz \cg$, can be calculated in terms of (i) cohomology groups of the $\gG_i$ and (ii) the cohomology groups of $X(G)$ and various subcomplexes of $X(G)$ (cf.~\cite{dbook,ddjo1,ddjo2,ddjmo,do10,jm}).  With trivial coefficients, the (co)homology groups of $\cg$ depend only on the $f$-vector of $X(G)$ (that is, the number of simplices of $X(G)$ in any given dimension) and the (co)homology groups of the $\gG_i$.

The {\it edge-independent random graph} is the probability space $G(n,p)$, defined as follows.
For a real number $0 \le p \le 1$ and natural number $n$, $G(n,p)$ is the set of all graphs on vertex set $[n]$ with probability measure defined by 
$$\prob (G) = p^{e_G} (1-p)^{{n \choose 2} - e_G},$$
where $e_G$ denotes the number of edges in $G$.  It can be viewed as the result of $n \choose 2$ independent coin flippings
, i.e., $G(n,p)$ is the probability space of all graphs on vertex set $[n]$ where each edge is included with uniform probability $p$, jointly independently.  \footnote{This is sometimes called the ``Erd\H{o}s--R\'enyi'' random graph, even though Erd\H{o}s and R\'enyi were interested in a different but closely related model, $G(n,m)$.}
The \emph{random flag complex with edge probability $p$}  is $X(n,p):=X(G(n,p))$.  In other words, it is the same probability space as $G(n,p)$ except that its elements are regarded as flag complexes rather than graphs. Similarly, the \emph{random graph product for $\ug$} is the group $\cg(n,p,\ug)$ associated to $G(n,p)$ and $\ug$.  

The groups $\cg(n,p,\ug)$ were considered previously by Charney--Farber \cite{chf}.  Somewhat earlier, Costa--Farber \cite{cof} had looked at  the special case of the random right-angled Artin group $A_{G(n,p)}$.  A formula for the cohomological dimension of $A_{G(n,p)}$  ($=1+\dim X(n,p)$) in terms of $(n,p)$ can be found in \cite{cof},  as well as, a formula for the ``topological complexity'' of its classifying space.  It is noted in \cite{chf} that if each $\gG_i$ is finite, then the graph product $\cg(n,p,\ug)$ is word hyperbolic if and only if $G(n,p)$ has no empty (induced) $4$-cycles; furthermore, it is determined when this condition holds ``with high probability.'' 

We will write $G\sim G(n,p)$ to mean that $G$ is chosen according to the distribution $G(n,p)$. 

In random graph theory one often lets $p$ depend on $n$.  For a given sequence $p = p(n)$, a graph property $\mathcal{Q}$ is said to hold  \emph{with high probability} (abbreviated $\aas$) if $$ \prob [G(n,p) \in \mathcal{Q}] \to 1$$ as $n\to \infty$.

We will use Bachmann--Landau and related notations.  Big $O$ and little $o$ are standard.  We also use $\Omega$ and $\omega$, defined as follows:  $f = \Omega(g) $ if and only if $g = O(f)$, and $f = \omega (g) $ if and only if $g= o(f)$. Whenever we use asymptotic notation such as big $O$ or little $o$, it is  understood to be as the number of vertices $n \to \infty$.  In slightly nonstandard notation, we will write $f \ll g$ if there exists a constant $\epsilon > 0$ such that $ f/ g = o(n^{-\epsilon})$ (in other words, $n^\epsilon f/g \to 0$ as $n\to \infty$).

A standard result in random graph theory is that if $d$ is a fixed positive integer and $$\omega \left( \frac{1}{n^{2/d}} \right) \le p \le o \left( \frac{1}{n^{2/(d+1)}} \right), $$ then $G(n,p)$ $\aas$ has cliques of order $d+1$ but not of order $d+2$. In other words, $X(n,p)$ is $\aas$ $d$-dimensional.  

A fundamental result of Erd\H{o}s-R\'enyi is that if $G\sim G(n,p)$ where $$p \ge \frac{\log n + \omega(1)} {n},$$ then $G$ is connected $\eaas$.  This result was generalized to higher dimensions by the second author in \cite{clique}, \cite{k12}.  Roughly, for the
random flag complex $X\sim X(n,p)$ of dimension $d$, we have $\aas$ that the reduced (co)homology, $\wH_*(X;\qq)$, is concentrated in degree $\dtwo$ (where $\lfloor x \rfloor$ means the greatest integer not exceeding $x$).  Moreover, with integer coefficients,
$H_i(X)=0$ for $i \le \lfloor (d-2)/4 \rfloor$ and $i > \lfloor d/2 \rfloor$.  (Our convention is that, when not specified, the coefficients of  (co)homology groups are assumed to be in $\zz$.)
In \S\ref{s:rflag} we strengthen these results by showing that the same is true $\aas$ for the homology of the ``punctured complex'' $X-\gs$  for all simplices $\gs$ of $X$.  (Here $X-\gs$ means the full subcomplex of $X$ spanned by all vertices which are not in $\gs$.)

Calculations of the cohomology of a graph product $\cg= \cg(G,\ug)$ with coefficients in its group ring or its group von Neumann algebra were done in  \cite{ddjo1,ddjo2,do10}.  An interesting feature is that there are essentially two different formulas depending on whether all $\gG_i$ are finite or all are infinite.  (In the mixed case the formulas are more complicated.)  When all $\gG_i$ are finite the formulas are in terms of the subcomplexes $X(G)-\gs$, where $\gs$ ranges over the simplices of $X(G)$ (including the empty simplex).  
These formulas are recalled in Propositions~\ref{p:gpringfin} and \ref{p:finL2} in \S\ref{ss:gpring} and \S\ref{ss:L2betti} below.  When all the $\gG_i$ are infinite, different formulas are needed, cf.~\cite{do10}. These formulas are expressed in terms of $H^*(\Lk(\gs,X(G)))$ and cohomology groups of the $\gG_i$ with appropriate coefficients. (Here $\Lk(\gs,X(G))$ denotes the link of a simplex $\gs$ in $X(G)$.)  The precise formulas are recalled in Propositions~\ref{p:gpringinf} and \ref{p:infL2} below.

This paper is organized as follows.  In \S\ref{s:coh} we review the formulas for the cohomology of graph products of groups.  In \S\ref{s:rflag} we review the results of \cite{clique, k12} on the cohomology of $X\sim X(n,p)$. Finally, 
in \S\ref{s:calc}, these results are combined to get fairly complete computations for the cohomology of random graph products of groups, $\cg:=\cg(G(n,p),\ug)$.  Beginning in \S\ref{s:rflag} we fix an integer $k\ge 0$ and impose the condition, 
	\begin{equation}\label{e:inequality}
	\frac{1}{n^{1/k}}\ll p \ll \frac{1}{n^{1/(k+1)}}. 
	\end{equation}
This condition entails $\dim X(n,p)=2k$ or $2k+1$, $\eaas$.  A striking consequence of our calculations is the following.

\begin{Theorem}\textup{(cf.~Theorem~\ref{t:Gfin}\,(\ref{i:middle})).}  Suppose $n^{-1/k}\ll p \ll n^{-1/(k+1)}$, for some given integer $k\ge 0$. Let $\cg=\cg(n,p,\ug)$ be a random graph product of finite groups.  Then $\aas$ $H^i(\cg;\qq\cg)$ is nonzero only for $i=k +1$ (where $k$ is the middle dimension of the random flag complex $X\sim X(n,p)$).  In other words, $\cg$ is a duality group over $\qq$ of formal dimension $k +1$.
\end{Theorem}
When all $\gG_i$ are infinite, different formulas establish the vanishing $\aas$ of $H^i(\cg;\qq \cg)$ for
$i<k +1$. (However, in degrees $> k+1$ the rational cohomology  can be nonzero.) This gives the following result.

\begin{Theorem}\textup{(cf.~Theorem~\ref{t:infinite}\,(1)).} 
Suppose $n^{-1/k}\ll p\ll n^{-1/(k+1)}$ for a given integer $k\ge 0$ and that $\cg$ is a random graph product of infinite groups. Then $\aas$ $H^i(\cg;\qq \cg)=0$ for $i<k +1$ and $H^{k+1}(\cg;\qq \cg)\neq 0$
\end{Theorem}

The first theorem applies to random right-angled Coxeter groups, the second to random right-angled Artin groups.

In either case (where all $\gG_i$ are finite or all are infinite), similar calculations give $\aas$ the virtual cohomological dimension of $\cg$, the number of its ends and, at least in some cases, its $L^2$-Betti numbers.  For example, in the case of random right-angled Artin groups we have the following.

\begin{Theorem}\textup{(cf.~Corollary~\ref{cor:raag}\,(3)).}
Suppose $n^{-1/k}\ll p \ll n^{-1/(k+1)}$, for a given integer $k\ge 0$.  Let $A_G$ be the random right-angled Artin group associated to $G\sim G(n,p)$. Then $\aas$ $\ltwob_i(A_G)$ is nonzero if and only if $i=k+1$.
\end{Theorem}
Our thanks go to the referee for some helpful comments.

\section{Cohomology of graph products}\label{s:coh}
\subsection{The $f$ and $h$ polynomials}\label{ss:fh}
Let $[n]:=\{1,\dots , n\}$.  Suppose $X$ is a simplicial complex on vertex set $[n]$.  We identify a simplex $\gs$ with its vertex set.  Following common practice, we shall blur the distinction between a simplicial complex as a poset of simplices or as a topological space and write $X$ for either.  By convention, the empty set is considered a simplex in any simplicial complex.  
Given $\gs \in X$, its \emph{link}, denoted $\Lk(\gs,X)$  (or sometimes simply $\Lk(\gs)$), is the simplicial complex whose poset of nonempty simplices is isomorphic to $X_{>\gs}$ ($:=\{\gt\in X \mid \gt >\gs\}$).

Let  $\cp(I)$ denote the power set of a finite set $I$. Given an $I$-tuple $\bt=(t_i)_{i\in I}$ of indeterminates and $J\in \cp(I)$, define a monomial $\bt_J$ by
	\begin{equation}\label{e:not}
	\bt_J=\prod_{j\in J} t_j
	\end{equation}
The \emph{$f$-polynomial} of $X$ is the polynomial in $\bt=(t_i)_{i\in \n}$ defined by
	\begin{equation*}\label{e:f}
	f_X(\bt):=\sum_{\gs\in X} \bt_{\gs} .
	\end{equation*}
The \emph{$\ah$-polynomial} of $X$ is defined by
	\begin{equation}\label{e:ah}
	\ah_X(\bt):=(\bone-\bt)_\n \,f_X\left(\frac{\bt}{\bone -\bt}\right),
	\end{equation}
where $\bone$ denotes the constant $n$-tuple $(1)_{i \in \n}$.  
If $\bt$ is the constant indeterminate given by $t_i=t$, then $f_X$ is a polynomial in one variable.   Denote it by $f_X(t)$.   If $\dim X= d$, then 
	\begin{equation*}\label{e:fsingle}
	f_X(t)=\sum_{i=-1}^{d} f_i(X) t^{i+1}.
	\end{equation*}
where $f_i(X)$ is the number of $i$-simplices in $X$ (and $f_\minus(X)=1$, the number of empty simplices).  The \emph{$h$-polynomial} of $X$ is then defined by  
	\begin{equation*}\label{e:hsingle}
	h_X(t):= \ah_X(t)/(1-t)^{n-d-1}=(1-t)^{d+1} f_X\left(\frac{t}{1-t}\right)\,.
	\end{equation*}
	
\subsection{(Co)homology of polyhedral products}\label{ss:pp}
As before,  $X$ is a simplicial complex with vertex set $[n]$.  Suppose $\uab = \{(A_i,B_i)\}_{i\in [n]}$ is a collection of pairs of nonempty subspaces.  For a point $x$ in the  Cartesian product, $\prod_{i=1}^n A_i$,  put $\gs(x):=\{i\in [n]\mid x_i \in A_i -B_i\}$.  The \emph{polyhedral product}, $\zx\uab$, is defined by
	\begin{equation*}
	\zx\uab:= \{x\in \prod_{i=1}^n A_i \mid \gs(x)\in X\}.
	\end{equation*}
($\gs(x)=\emptyset$ is allowed.)  
When all the $(A_i,B_i)$ are all equal to the same pair $(A,B)$, we write $\zx(A,B)$ for the polyhedral product.  The (co)homology of these spaces can be calculated.  The formulas simplify if either 1) each $B_i$ is contractible (e.g., if $B_i$ is a base point $*_i$) or 2) each $A_i$ is contractible (cf.~\cite{bbcg2}).  
If each $B_i$ is contractible, then
	\begin{equation}\label{e:pp0}
	\wH_*(\zx\uab)=\bigoplus_{\gs\in X} \wH_*(\wh{\ua}^{\gs}),
	\end{equation}
where $\wh{\ua}^{\gs}$ denotes the $\gs$-fold smash product of the $A_i$.  (See \cite[Thm.~2.15]{bbcg2}.) By using the K\"unneth Formula, the (co)homology of $\wh{\ua}^{\gs}$ can be calculated from that of the $A_i$.  The 
formula  is simplified if we take with coefficients in a field $\bF$.
Using \eqref{e:pp0}, we see  that there is an isomorphism of algebras:
	\begin{equation}\label{e:pp1}
	H^*(\zx\uab;\bF)=\left[\bigotimes_{i=1}^{m} H^*(A_i;\bF)\right] / \,\ci (X),
	\end{equation}
where $\ci(X)$, the \emph{generalized Stanley-Reisner ideal}, is the ideal in the tensor product of algebras generated by all $x_{i_1}\otimes\cdots \otimes x_{i_l}$, such that $x_{i_k}\in \rh^*(A_{i_k};\bF)$ and such that  $\{i_1, \dots ,i_l\}$ is not a simplex of $X$
(\cite {densuc} or \cite[Thm.~2.34]{bbcg2}).  The right hand side of \eqref{e:pp1} is the \emph{generalized face ring}.

On the other hand, when each $A_i$ is contractible the formula is
	\begin{equation}\label{e:pp2}
	H^*(\zx\uab)=\bigoplus_{\substack{I\le [n]\\I \text{ is not a simplex of }X}} H^*(X(I)*\wh{\ub}^I),
	\end{equation}
where $X(I)$ denotes the full subcomplex spanned by $I$, $\wh{\ub}^I$ denotes the $I$-fold smash product of copies of the $B_i$, and $X(I)*\wh{B}^I$ denotes their join. (Again, each summand on the right hand side can be computed from the K\"unneth Formula.) 
If each $B_i$ is connected and simply connected, then the fundamental group of $\zx\uab$ is the graph product $\cg(G;\ug)$, where the graph $G$ is the $1$-skeleton of $X$ and $\gG_i=\pi_1(A_i)$ (cf.~\cite{d11}).  If each $\gG_i$ is infinite, then the argument of \cite{do10} shows
	\begin{equation}\label{e:gr1}
			\grh^m(\zx\uab;\zz \cg)=\bigoplus_{\substack{\gs\in  X\\
			i+j=m}} H^{i}(\cone \Lk(\gs),\Lk(\gs); H^j(\ua^{\gs};\zz \cg)),
	\end{equation}
Here $\GR$ means the  ``associated graded'' group (because each summand in  \eqref{e:gr1} is the $E^{i,j}_\infty$ term of a spectral sequence). Also, $\ua^{\gs}$ stands for the $\gs$-fold product $\prod_{i\in \gs} A_i$ so that the coefficients in a summand on the right hand side of \eqref{e:gr1} can be calculated from the K\"unneth Formula.  Indeed,  once we replace $\zz$ by a field $\bF$, we get
\[
H^*(\ua^{\gs};\bF \cg))=\left[\bigotimes_{i\in \gs} H^*(A_i;\bF \gG_i)\right]\otimes _{\gG_i} \bF\cg .
\]

\subsection{Polyhedral products as classifying spaces for  graph products}
Our interest in the polyhedral product construction stems from its relationship to graph products of groups.  Given a graph $G$ with vertex set $[n]$ and a collection of discrete groups $\ug=\{\gG_i\}_{i\in \nn}$, let $\cg$ ($=\cg(G,\ug)$) denote their graph product. For any  subset $I\le \n$, let $\ug^I$ denote the ordinary product, $\prod_{i\in I} \gG_i$.  
Let $B\gG_i$ denote the classifying space for $\gG_i$ (i.e., $B\gG_i$ is a $K(\gG_i,1)$ complex).  We consider two cases:  $(A_i, B_i)=(B\gG_i,*_i)$ (which we denote $(B\ug,\ul{*})$ and $(A_i,B_i)=(\cone \gG_i,\gG_i)$ (denoted  $(\cone \ug, \ug)$).

\begin{proposition}\label{p:pp}
\textup{(\cite{do10}).}   
Suppose, as above, $G$ is a graph with vertex set $[n]$, $X(G)$ is its flag complex, and $\cg$ is the graph product of the $(\gG_i)_{i\in [n]}$.
\begin{enumerate1}
\item
$B\cg = \zxg (B\ug,\ul{*})$.
\item
Let $\cg_0$ denote the kernel of the natural map $\cg \to \ug^{[n]}$ to the direct product.  Then $B\cg_0=\zxg(\cone \ug, \ug)$.
\end{enumerate1}
\end{proposition}

\begin{proof}[Sketch of proof]
One first proves (2).  The group $\ug^{[n]}$ acts on $\zxg(\cone \ug, \ug)$ and $\cg$ can be identified with the group of all lifts of elements in $\gG^{[n]}$ to the universal cover $\tzxg(\cone\ug, \ug)$.  Since $X(G)$ is a flag complex, $\tzxg(\cone \ug,\ug)$ is the standard realization of a right-angled building (cf.~\cite[Prop.\,2.10]{d11}) and hence, is contractible.  Since $\cg_0$ is the group of covering transformations, statement (2) follows.  To prove (1), first observe that $\zxg(\cone \ug, \ug)$ is homotopy equivalent to the covering space of $\zxg (B\ug,\ul{*})$ corresponding to the subgroup $\cg_0$.  Next observe that $\zxg(E\ug,\ug)$ is homotopy equivalent to $\zxg(\cone \ug, \ug)$, where $E\gG_i$ is the universal cover of $B\gG_i$ and $E\ug:=\{E\gG_i\}_{i\in [n]}$.  Hence, the universal cover of $\zxg(E\ug, \ug)$ is also contractible and so, can be identified with $E\cg$, which proves (1).
\end{proof}

\subsection{Homology with trivial coefficients}\label{ss:trivial}
Notation is as before.  
Given a subset $I\le \n$ and  a field $\bF$, the dimension of the following tensor product in degree $m$ is denoted by 
	\begin{equation*}\label{e:bm}
	b_{I,m}(\ug;\bF):=\dim_\bF \big(\bigotimes_{i\in I} \rh^*(B\gG_i;\bF)\big)^m.
	\end{equation*}
In other words, $b_{I,m}(\gG;\bF)$ is the $m^{th}$ Betti number of the smash product of the  $B\gG_i$, $i\in I$.  When $\ug$ is the constant sequence $\gG_i=\gG$ and $k\in \nn$, put $b_{k,m}(\gG;\bF):=b_{[k],m}(\ug;\bF)$.  In the next proposition we use \eqref{e:pp0} and Proposition~\ref{p:pp} to compute the Betti numbers of $B\cg$.
\begin{proposition}\label{p:bcg1}
Let $b_m(B\cg;\bF):=\dim_\bF H_m(B\cg;\bF)$ be the $m^{th}$ Betti number of $B\cg$.  Then
	\[
	b_m(B\cg;\bF)=\sum_{\gs\in  X}b_{\gs,m}(\ug;\bF).
	\]
In particular, if $\ug$ is the constant sequence $\gG$, then
	\[
	b_m(B\cg;\bF)=\sum_{k=1}^m f_{k-1}b_{k,m}(\gG;\bF),
	\]
 where $f_{k-1}=f_{k-1}(X)$ is  the number of $(k-1)$-simplices in $X$.
\end{proposition}

For example, if $\gG=\zz/2$, then $\cg(G;\zz/2)=W_G$, the right-angled Coxeter group associated to $G$ and $BW_G=\zxg(B(\zz/2),*)$.  Let $\bF_2$ be the field with 2 elements.  Since $H^*(\zz/2;\bF_2)$ is the polynomial ring $\bF_2[t]$, formula \eqref{e:pp1} and Proposition~\ref{p:pp} give the following result of \cite{dj91},
	\begin{equation*}\label{e:wg}
	H^*(BW_G;\bF_2)=\bF_2[X],
	\end{equation*}
where the right hand side denotes the Stanley-Reisner face ring of $X$.  It follows that the Poincar\'e series, $\sum b_i(BW_G;\ftwo) t^i$, is given by
	\begin{align}
	\sum_{i=0}^\infty b_i(BW_G;\ftwo) t^i &=
	\sum _{i=-1}^{d}\frac{f_i(X)t^{i+1}}{(1-t)^{i+1}}\notag\\
	&=f_X\left(\frac{t}{1-t}\right):=\frac{h_X(t)}{(1-t)^{d+1}},\label{e:cox}
	\end{align}
where $d=\dim X$.  

For another example, if $\gG=\zz$, then $\cg=A_G$, the right-angled Artin group associated to $G$.  Since $B\zz=S^1$,  Proposition~\ref{p:bcg1} yields
	\begin{equation}\label{e:ag}
	b_k(A_G;\bF)=f_{k-1}(X)
	\end{equation}
and this implies that $H^*(BA_G)=\bigwedge [X]$, the exterior face ring of $X$, cf.\ \cite{cd05}, \cite{kr}. Alternatively, we could have proved this (even with integral coefficients) by using formula \eqref{e:pp1} and Proposition~\ref{p:pp} as before.

\paragraph{Some definitions.}
A group $\gG$ is \emph{type F} if $B\gG$ has a model which is a finite CW complex.  If $\gG$ is type F, then it is 
automatically \emph{type FL},  which means that $\zz$ has a finite resolution by finitely generated free $\zz G$-modules. $\gG$ is \emph{type FP} if $\zz$ has a finite resolution by finitely generated projective $\zz G$-modules.  Similarly, for a commutative ring $R$, $\gG$ is type $\mathrm{FL}_R$ (resp. $\mathrm{FP}_R$) if $R$ has a finite resolution by finitely generated, free (resp. projective) $R\gG$-modules.
$\gG$ is \emph{virtually torsion-free} if it has a torsion-free subgroup $\gG_0$ of finite index.  A virtually torsion-free group $\gG$ is, respectively,  \emph{type VF, VFL or VFP} as $\gG_0$ is F, FL or FP.

If each $\gG_i$ is finite of order $q_i+1$, then we say $\ug$ \emph{has order $\bq+\bone$}, where
\(
\bq:=(q_i)_{i\in \nn}
\).
If $\cg_0$ denotes the subgroup of $\cg(G,\ug)$ defined in Proposition~\ref{p:pp}\,(ii), then $\cg_0$ is a torsion-free subgroup of finite index in $\cg$.  (In the notation \eqref{e:not} from \S\ref{ss:fh}, its index is $(\bone +\bq)_\n$.) 
By Proposition~\ref{p:pp}\,(ii), $B\cg_0=\zx(\cone \ug, \ug)$, which is a finite complex.  So, $\cg$ is type VF.  Applying \eqref{e:pp2}, we get the following.

\begin{proposition}\label{p:cg0}
Suppose each $\gG_i$ is finite and $\ug$ has order $\bq+\bone$.  Then 	\[
	H_*(B\cg_0)=\bigoplus_{\substack{I\le [n]\\I\notin  X}}H_*(\cone X(I), X(I)) \otimes M_I,
	\]
where $M_I$ is a free abelian group of rank $\bq_I$. (It is the ``Steinberg module'' for $\ug^I$, i.e., the $I$-fold tensor product of augmentation ideals of $\zz \gG_i$, $i\in I$.)
\end{proposition}

\begin{Remark}
If $G$ is not a complete graph, then there are distinct elements $i$, $j$ in $\n$ which are not connected by an edge; so,  $H_1(\cone X(I),X(I))=\zz$, for $I=\{i,j\}$. It follows that when $\cg_0$ is nontrivial, its abelianization maps onto $\zz$.  So, Proposition~\ref{p:cg0} implies that a graph product of finite groups never has Kazhdan's property T unless it is  finite.
\end{Remark}

Using \cite{dpe} we can compute the Euler characteristic of $\zxg(\cone \ug, \ug)$ as well as the ``orbihedral Euler characteristic'' of $\zxg(\cone \ug, \ug)/\ug^{[n]}$ (also called the \emph{rational Euler characteristic}, $\chi(\cg)$, of $\cg$).

\begin{proposition}\label{p:pe}\textup{(cf.~\cite{dpe}).} 
Suppose each $\gG_i$ is finite and $\ug$ has order $\bq+\bone$.  
\begin{enumerate1}
\item
The Euler characteristic of $B\cg_0$ is given by
	\begin{equation*}
	\chi(B\cg_0)=(\bq+\bone)_\n f_{X(G)}\left(\frac{-\bq}{\bq+\bone}\right)=\ah_{X(G)}(-\bq).
	\end{equation*}
\item	
The rational Euler characteristic of $\cg$ is given by
	\begin{equation*}
	\chi(\cg)=\frac{\chi(B\cg_0)}{(\bq+\bone)_\n} 
	=f_X\left(\frac{-\bq}{\bq+\bone}\right).
	\end{equation*}
\end{enumerate1}
\end{proposition}
\begin{proof}
The formula in (1) is proved in \cite[Cor.~2]{dpe}.  The group $\cg_0$ has index $(\bq+\bone)_\n$ in $\cg_0$; so, (2) is immediate from the definition of the rational Euler characteristic.
\end{proof}

Recall that if a group is nontrivial and type FL, then it is necessarily infinite.
\begin{proposition}\label{p:euler1}
Suppose each $\gG_i$ is type FL (so that its Euler characteristic is defined). Let $e_i=e(\gG_i):=\chi(\gG_i)-1$ be the reduced Euler characteristic of $B\gG_i$, and put $\be=(e_i)_{i\in \nn}$.  Then
	\(
	\chi(\cg)=f_X(\be).
	\)
\end{proposition}

\begin{proof}
By Proposition~\ref{p:pp}\,(1), $B\cg = \zxg (B\ug,\ul{*})$.  In \cite[Cor.\,1]{dpe} there is a formula for the Euler characteristic of the polyhedral product, which gives $\chi(B\cg)=f_X(\be)$.
\end{proof}

\subsection{Cohomology with group ring coefficients}\label{ss:gpring}
An important invariant of an infinite discrete  group $H$ is its cohomology with coefficients in its group ring, $\zz H$.  For example, the number of ends of $H$, denoted by $\Ends H$, is $1$, $2$ or $\infty$ as the rank of $H^1(H;\zz H)$ is $0$, $1$ or $\infty$.  If $H$ is type $\mathrm{FP}_R$, then its cohomological dimension, $\cd_R(H)$, with respect to a commutative ring $R$ is given by,
	\begin{equation*}\label{e:Fcd}
	\cd_R H=\max\{k\mid H^k(H;R H)\neq 0\}.
	\end{equation*}
As usual, when $R=\zz$, the subscript is omitted  and we write $\cd H$ instead of $\cd_R H$.

\paragraph{The case where each $\gG_i$ is finite.}	
In what follows $X=X(G)$ and for any $\gs\in  X$, $X-\gs$ means the full subcomplex of $X$ spanned by $\n-\gs$.

\begin{proposition}\label{p:gpringfin}\textup{(\cite{ddjo2} or \cite[Cor.~9.4]{ddjmo})}.
Suppose each $\gG_i$ is finite. Then, for $\cg=\cg(\ug),G)$, 
		\[
			H^*(\cg;\zz \cg)=\bigoplus_{\gs\in  X} H^*(\cone X,X-\gs)\otimes \hat{A}^{\gs},
		\]
where $\hat{A}^{\gs}$ is a certain (free abelian) subgroup of $\zz (\cg/\ug^{\gs})$ (where $\ug^{\gs}$ denotes  the $\gs$-fold product of the $\gG_i$). 
\end{proposition}

\begin{corollary}\label{c:endsvcd}
Suppose each $\gG_i$ is finite.
\begin{enumerate1}
\item If $X$ is a simplex, then $\cg$ is finite and $\Ends(\cg)=0$.  If $X$ is the suspension of a simplex and the groups for both suspension vertices are $\cong \zz/2$, then $\Ends (\cg)=2$.  Otherwise,
\[
\Ends (\cg)=
\begin{cases}
1, 	& \text{if $\rh^0(X-\gs)=0$ for all $\gs\in  X$;}\\
\infty,	& \text{if $\rh^0(X-\gs)\neq 0$ for some $\gs\in  X$.}
\end{cases}
\]
\item
\[
\vcd \cg = \max\{k\mid H^{k-1}(X-\gs)\neq 0 \text{ for some }\gs\in  X\}.
\]
\end{enumerate1}
\end{corollary}

For example, when $W_G$ is the right-angled Coxeter group associated to the graph $G$, Proposition~\ref{p:gpringfin} becomes the following formula of \cite{d98},
	\begin{equation*}\label{e:racg}
	H^*(W_G;\zz W_G)=\bigoplus_{\gs\in  X} H^*(\cone X,X-\gs)\otimes \zz W^{\gs},
	\end{equation*}
where $W^{\gs}$ denotes the set of elements in $W$ which can end (exactly) with  letters of $\gs$ and where $\zz W^{\gs}$ denotes the free abelian group on $W^{\gs}$.

\paragraph{The case where each $\gG_i$ is infinite.}  In what follows $\grh^*(\ ;\ )$ means the associated graded group arising from a certain filtration.

\begin{proposition}\label{p:gpringinf}\textup{(\cite[Thm.~4.5]{do10})}. 
Suppose each $\gG_i$ is infinite.  Then for $\cg=\cg(\ug,G)$, we have
		\[
			\grh^m(\cg;\zz \cg)=\bigoplus_{\substack{\gs\in  X\\
			i+j=m}} H^{i}(\cone \Lk(\gs),\Lk(\gs); H^j(\ug^{\gs};\zz \cg)).
		\]
\end{proposition}

For example, if  $A_G$ is the right-angled Artin group associated to $G$, we have the following formula of \cite{jm} and \cite{do10}

		\begin{equation}\label{e:raag}
		\grh^n(A_G;\zz A_G)=\bigoplus _{\gs\in  X} H^{n-\dim\gs -1}(\cone \Lk(\gs),\Lk(\gs)) \otimes H^{\dim\gs+1}(\zz^{\gs};\zz A_G),
		\end{equation}
where $\zz^{\gs}$ denotes the free abelian group on $\gs$ and $\dim\gs+1$ is the number of elements in $\gs$ (so that $H^{\dim\gs+1}(\zz^{\gs};\zz A_G)=\zz(A_G/\zz^{\gs}$)).

\subsection{\texorpdfstring{$L^2$}{L2}-Betti numbers}\label{ss:L2betti}
Let $W_G$ be the right-angled Coxeter group associated to a graph $G$.  Its growth series, $W_G (\bt)$, is the rational function in $\bt =(t_i)_{i\in \n}$ given by
	\begin{equation*}\label{e:growth}
	\frac{1}{W_G(\bt)}=f_{X(G)}\left(\frac{-\bt}{\bone+\bt}\right) =\frac{\hat{h}_{X(G)}(-\bt)}{(\bone +\bt)_{[n]}},
	\end{equation*}
(See \cite[\S17.1]{dbook}.) 

Let $\car_G$ denote the region of convergence of $W_G(\bt)$.  For example, if $G=V[n]$, the graph with vertex set $[n]$ and no edges, we have
	\begin{equation*}\label{e:noedge}
	\frac{1}{W_{V[n]}(\bt)}=1-\sum_{i=1}^{n} \frac{t_i}{1+t_i}.
	\end{equation*}
It follows that
	\begin{equation}\label{e:noedge2}
	\car_{V[n]}\cap [0,\infty)^n=\{\bt\in [0,\infty)^n\mid \sum_{i=1}^n \frac{t_i}{1+t_i} < 1\}.
	\end{equation}
(Indeed,  for $\bt$ in the indicated range $1/W_{V[n]}(\bt)$ is always positive; hence, $W_{V[n]}(\bt)$ converges.)

For another example, when $\bt$ is the constant indeterminate $t$, we have 
	\[
	\frac{1}{W_G(t)}=\frac{h_{X(G)}(-t)}{(1+t)^{d+1}},
	\]
so that $\car_G$ consists of all complex numbers of modulus less than the smallest positive real root $\gr$ of $h_{X(G)}(-t)$. (Note $\gr\in (0,1]$.)

In \S\ref{ss:Gfinite} we will need the following lemma.

\begin{lemma}\label{l:car}
Suppose $G$, $G'$ are two graphs with the same vertex set $[n]$ such that $G'$ is obtained by deleting edges of $G$.  Then $\car_{G'}\le \car_G$.  In particular, for any graph $G$, $\car_{G}$ always contains the region defined by \eqref{e:noedge2}.
\end{lemma}

\begin{proof}
Since there are more relations in $W_G$ than in $W_{G'}$, the number of elements of word length $k$ with letters in a given subset of $[n]$ is greater for $W_{G'}$ than for $W_G$.  Hence, the coefficients in the power series $W_{G'}(\bt)$ are positive integers which dominate the coefficients of $W_G(\bt)$.  So, $\car_{G'}\le \car_G$.  The last sentence of the lemma follows immediately.
\end{proof}

Let $\bone/\bt$ denote the sequence $(1/t_i)_{i\in \nn}$.
For each simplex $\gs\in  X$, define a series
	\begin{equation*}\label{e:dgs1}
	D_\gs(\bt)=\sum_{\gt\in  X_{\ge \gs}} \frac{(-1)^{\dim\gt-\dim\gs}}{(\bone+\bone/\bt)_{I(\gt)}}\,.
	\end{equation*}
Let $G(\gs)$ denote the $1$-skeleton of $\Lk(\gs)$.  Notice that $D_\gs(t)$ is related to the power series for $W_{G(\gs)}$ by the following formula (see \cite[Lemma~17.1.8, Cor.~20.6.17]{dbook}).
	\begin{equation}\label{e:dgs2}
	D_\gs(\bt)=\frac{\bone}{(\bone+\bt)_{\gs}}\cdot \frac{1}{W_{G(\gs)}(\bone/\bt)}.
	\end{equation}

\begin{proposition}\label{p:finL2} \textup{(cf.\ \cite[Thm.~20.8.4]{dbook}).}
Suppose each $\gG_i$ is finite and $\ug$ has order $\bq+\bone$.  Suppose further  that $\bone/\bq$ lies in the region of convergence $\car_G$ for $W_G(\bt)$.  Then 
	\[
	\ltwob_m(\cg) =\sum_{\gs\in  X} b_m(\cone X, X-\gs;\qq)\cdot D_\gs (\bq),
	\]
where $b_m(\cone X,X-\gs;\qq)$ is the ordinary Betti number (with rational coefficients) of the pair. (Since $\cone X$ is contractible,  $b_m(\cone X,X-\gs;\qq)$ is equal to the reduced Betti number $\tilde{b}_{m-1}(X-\gs;\qq)$.) 
\end{proposition}

As one might suspect from the results in the previous subsection, the calculation is different when all $\gG_i$ are  infinite.  So, suppose each $\gG_i$ is infinite and that their $L^2$-Betti numbers are defined.  Given $\gs\in  X$, let $\ltwob_{\gs,m}$ denote the $m^{th}$ $L^2$-Betti number of the $\gs$-fold product, $\gG^{\gs}$.  If $\gs=\{i_1,\dots, i_k\}$, then, by the K\"{u}nneth Formula,
	\begin{equation}\label{e:bgs}
	\ltwob_{\gs,m}=\sum_{f(i_1)+\cdots+f( i_{k})=m} \ltwob_{f(i_1)}(\gG_{i_1})\cdots \ltwob_{f(i_{k})}(\gG_{i_k}).
	\end{equation}
where $f$ ranges over all functions from $\gs$ to $\nn$ which sum to $m$.

\begin{proposition}\label{p:infL2} \textup{(\cite[Thm.\,4.6]{do10}).}
Suppose each $\gG_i$ is infinite.  Then
		\[
			\ltwob_l(\cg)=\sum_{\substack{\gs\in  X\\
			i+m=l}} b_i(\cone \Lk(\gs),\Lk(\gs))\cdot \ltwob_{\gs,m},
		\]
where $ \ltwob_{\gs,m}$ is given by \eqref{e:bgs}.
\end{proposition}

Since all $L^2$-Betti numbers of the infinite cyclic group vanish, for right-angled Artin groups  the previous proposition drastically simplifies to the following.

\begin{corollary}\label{c:raag}\textup{(Davis-Leary \cite{dl})}.
$\ltwob_l(A_G)=b _l(\cone X(G),X(G);\qq)$.  In other words, the $L^2$-Betti numbers of $A_G$ are the ordinary reduced Betti numbers of $X(G)$ with degree shifted up by $1$.
\end{corollary}

\section{Random flag complexes}\label{s:rflag}

In this section we state some results about the topology of the random flag complex $X=X(n,p)$. 
Earlier results were proved by the second author in \cite{clique, k12}.  Here we show that similar results hold $\aas$ for  $X-\gs$  for all simplices $\gs$ of $X$, and for $\Lk(\gs,X)$ for all simplices $\gs \in X$ of sufficiently small dimension.

\begin{theorem}\label{t:randombetti} \textup{(cf. \cite{k12, clique, hkp}).}
Suppose $X\sim X(n,p)$ where $$ \frac{1}{n^{1/k }} \ll p \ll \frac{1}{ n^{1/(k+1) }},$$ where $k$ is a given integer $\ge 0$. Then w.h.p., for every face $\sigma \in X$ the subcomplex $X - \sigma$ satisfies the following properties: 

\begin{enumerate1}
\item $\dim (X -\sigma) = d$, where $d=2k+1$ (when $\go(n^{-2/(2k+1)})\le p$) or  $d=2k$ (when $p\le o(n^{-2/(2k+1)})$).
\item $\wH_i(X - \sigma;\qq) = 0$ if and only if $i \neq k $.
\end{enumerate1}
\end{theorem}

\begin{remark}
The case $\sigma = \emptyset$ follows from \cite[Cor.~2.2]{k12}.
\end{remark}

\begin{remark}\label{r:integer}
As for homology with integer coefficients, it is proved in \cite{clique} that $\aas$ $\wH_i(X)$ vanishes whenever $i$ lies in either of the following two ranges,
\begin{enumeratea}
\item
$i\le \lfloor (k-1)/2 \rfloor$ or 
\item
$i>k$.
\end{enumeratea}
With regard to (a), it is proved in \cite{clique} that $X$ is $\lfloor (k-1)/2 \rfloor$-connected $\eaas$.  With some work, this can be extended to show that $X-\gs$ is $\lfloor (k-1)/2 \rfloor$-connected for all $\gs\in  X$. With regard to (b), with no additional work,  the argument in \cite{clique} shows that for any full subcomplex $Y$ of $X$, for $i>k$, $H_i(Y)=0$ $\eaas$.  In particular, this holds for $Y=X-\gs$.

We don't know if statement (2) of Theorem~\ref{t:randombetti} holds with integer coefficients when $(k-1)/2 <i\le k$.  In this range  $\wH_i(X-\gs)$ could have torsion (cf.\ the comments in Section 7 of \cite{k12}). In particular, $\wH_{k}(X-\gs)$ might have nontrivial torsion.  If this happens, then, by the Universal Coefficient Theorem, $\wH^{k+1}(X-\gs)$  has nontrivial torsion.
 \end{remark} 

\begin{remark} \label{r:notzero}
For each $i \ge 0$ there is a small interval of $p$ for which both $\wH_i(X)$ and $\wH_{i+1}(X)$ are nonvanishing.  For example, when $i=0$, it is well known that if $ c / n \le p \le o(\log n / n)$ \cite{Bollo}, then w.h.p.\ $G(n,p)$ is disconnected but contains cycles.  For every $i$, the width of this window of overlap is of order $$\Theta (( \log n / n)^{1/i+1}),$$
(where $f=\Theta(g)$ means $f=O(g)$ and $g=O(f)$).  Since this is peripheral to our main argument,  we do not prove it here.
\end{remark}
The main tool needed to prove Theorem \ref{t:randombetti} is Theorem \ref{t:garland} below. In \cite{garland} Garland proved vanishing results for cohomology groups of $k$-dimensional simplicial complexes (possibly with coefficients in a unitary representation of the fundamental group) through degree $k-1$ provided the link of each $(j-2)$-simplex $\gs$, with $j\le k$,  is connected and that its Laplacian in degree $0$ has sufficiently large spectral gap.  

Suppose $X$ is a pure simplicial complex of dimension at least $1$.  Given a vertex $v$, let $m(v)$ denote the degree of $v$ in the $1$-skeleton, $X^1$. The \emph{averaging operator} $A:C^0(X;\rr)\to C^0(X;\rr)$ and the \emph{normalized Laplacian} $\gD:C^0(X;\rr)\to C^0(X;\rr)$ are defined by
	\begin{equation*}
	A(\gf)(v):=\frac{1}{m(v)}\sum\gf(w)\quad \text {and}\quad
	\gD:=1-A,
	\end{equation*}
where the summation is over all vertices $w$ which are adjacent to $v$.  Then $\gD$ is positive semidefinite. The spectrum of $A$ lies in $[-1,1]$;  hence, the spectrum of $\gD$ lies in $[0,2]$.  Let $0=\gl_1\le \gl_2\le\cdots \le \gl_n$ be the eigenvalues of $\gD$.  $X$ is connected if and only if  $0$ occurs with multiplicity $1$.  Assuming this to be the case, the first positive eigenvalue, $\gl_2$, is called the \emph{spectral gap}.

Garland's method is explained and expanded upon in \cite{ballmannswiat}, where one can find the following result.  (See also \cite{borel}.)

\begin{theorem}\label{t:garland} \textup{(Ballmann-\'Swi\c atkowski \cite[Thm.~2.5]{ballmannswiat})} 
Suppose $X$ is a  finite simplicial complex and $k$ is a positive integer $< \dim X$ so that the $k$-skeleton, $X^k$, is pure (i.e., every $\gs  \in X^k$, is contained in at least one $k$-dimensional simplex).  Given $\gs\in X$, let $\gl_1(\gs)\le\gl_2(\gs)\le \cdots$, denote the eigenvalues of the normalized Laplacian on $C^0(\Lk(\gs,X);\rr)$.
Assume that there is an $\geps>0$ so that  $\gl_2(\gs) \ge \frac {k}{k+1} +\geps$. Then $H^{k-1}(X;\rr)=0$.
\end{theorem}
 
We need another tool before proving Theorem \ref{t:randombetti}, namely the following estimate from \cite{hkp} on spectral gaps of edge-independent random graphs.
 
 \begin{theorem}
\label{thm:ergap} 
Let $G\sim G(n,p)$ be a Bernoulli random graph. Let $\gD$ denote the normalized Laplacian of $G$, and let $\lambda_1 \leq \lambda_2 \leq \cdots \leq \lambda_n$ be the eigenvalues of $\gD.$  For every fixed $\alpha \geq 0$, there is a constant $\widetilde{C}_{\alpha}$ depending only on $\alpha$, so that if
$$p \ge  \frac{(\alpha+1)\log n + \widetilde{C}_{\alpha} \sqrt{\log n} \log\log n}{n}$$  
then $G$ is connected and 
$$\lambda_2 (\gG) > 1- o(1),$$
with probability $1 - o(n^{-\alpha})$.
\end{theorem}
 
 \begin{proof}[Proof of Theorem \ref{t:randombetti}]
 
The first claim is that $\dim (X -\sigma) = d$ for every simplex $\sigma \in X$.  When $\sigma = \emptyset$  this is a standard result about random graphs 
--- if $p$ is in the given regime, then w.h.p.\ there are $d$-simplices (i.e., cliques of order $d+1$) but no $(d+1)$-simplices (i.e., cliques of order $d+2$), which is exactly the claim.

We include a proof here of the case of an arbitrary $\sigma$ for the sake of completeness.
First consider the case $\gs=\emptyset$.    The claim that $X(n,p)$ is w.h.p.\ $d$-dimensional is equivalent to showing that $X(n,p)$ w.h.p.\ contains a simplex on $d+1$ vertices, but contains no simplices on $d+2$ vertices.  This is a special case of standard results on subgraphs of random graphs \cite{Bollo}.  We recall the proof here.

Let $f_{i-1}$ be the number of simplices on $i$ vertices.  
The expected value is given by 
	\begin{equation}\label{e:f}
	\expect[f_{i-1}] = { n \choose i} p^{i \choose 2}.
	\end{equation}
If $p  \ll n^{-2/(d+1)}$,  then  
\begin{align*}
\expect[f_{d+1}]  & = { n \choose d+2} p^{d+2 \choose 2}\\
& \le n^{d+2}  \left( n^{-2/(d+1)-\epsilon} \right)^{d+2\choose 2}\\
&=  n^{-c_1}, 
\end{align*}
where $c_1= \epsilon {d+2\choose 2} >0$.  By Markov's inequality, $f_{d+1} =0$ w.h.p.  It follows that $\dim X \le d$. 

On the other hand, if $p \gg  n^{-2 / d}$ then
	\begin{align*}
	\expect[f_d]  & = { n \choose d+1} p^{d+1 \choose 2}\\
	& \ge \frac{(1-o(1))}{(d+1)!} n^{c_2} \\
	\end{align*}
where $c_2= \epsilon {d+1 \choose 2} > 0 $.

%
Janson's inequality \cite{Alon} gives for this range of $p$ that
$$\prob[ f_d \le (1/2) \expect(f_d)] \le e^{-n^{c_2}/6}.$$
We can apply this argument separately to each of the subcomplexes $X-\sigma$.  
Since $X$ is w.h.p.\ $d$-dimensional, there are w.h.p. at most $O(n^{d+1})$ faces total.  Applying a union bound, the total probability that any one of these complexes fails to be $d$-dimensional is at most   
\[
O(n^{d+1}) e^{-n^{c_2}/6 }= o(1).
\]

\bigskip
 
For the second claim, that 
 $\wH_i(X - \sigma;\qq) = 0$ whenever $i \neq k$, we extend the ideas from \cite{clique} and \cite{k12} which were used to prove this in the case $\sigma = \emptyset$.  The proof has two parts:  first we check that  $\wH_i(X - \sigma;\qq) = 0$ when $i > k$ and then when $i < k$.

The proof that  $\wH_i(X;\qq) = 0$ when $i > k$ in \cite[Section 5]{clique} is to show first that for this range of $p$, homology is $\aas$ generated by cycles supported on simplices which are supported on a bounded number of vertices as $n \to \infty$, and then that all such cycles are boundaries.  
The same argument goes through verbatim to show that this also holds for every subcomplex of $X$.  In particular, $\wH_i(X - \sigma;\qq) = 0$ for every simplex $\gs$ and with $i > k$.

The proof that $\wH_i(X;\qq) = 0$ when $i < k$ in \cite{k12} uses Theorem \ref{t:garland}.
For any  $\gs\in  X$, write $\Lk(\gs)$ as short for $\Lk(\gs,X)$.   It is shown in \cite{k12} that, for this range of $p$, the $(k +1)$-skeleton of $X$ is w.h.p.\  pure, and  that w.h.p.\ for every $(k+1)$-simplex $\alpha \in X$, $\lambda_2 (\alpha)  > 1-o(1)$.  (Here we are considering the link of $\alpha$ in the $(k +1)$-skeleton of $X$, i.e., as a graph.)  In particular, all these graphs are connected.

We extend this proof to show that $\wH_i(X - \sigma;\qq) = 0$  for $i < k$ and for all $\gs \in X$ by applying Theorem \ref{t:garland} to each of the subcomplexes $X-\gs$.  The link of a codimension-$2$ face in the $(i+1)$-skeleton $X-\sigma$ is still a Bernoulli random graph, and we can use Theorem \ref{thm:ergap}.   Since $ n^{-1/k} \ll p$, the probability that any of these graphs have spectral gap $\lambda_2 < 1- \epsilon$ is $o(n^{-\alpha})$ for every fixed $\alpha > 0$.

On the other hand, w.h.p.\  $X$ is $d$-dimensional, where $d=2k or 2k+1$, so there are $O(n^{2k+2})$ simplices in total.  Applying a union bound, the probability that any of the polynomially many random graphs arising as the link of a simplex in a deleted subcomplex has small spectral gap tends to zero.  Then Theorem \ref{t:garland} gives that  w.h.p.\ $\wH_i(X - \sigma;\qq) = 0$ for every face $\sigma$ and $i < k$.
 \end{proof}

 We also need the following in \S\ref{ss:infinite}.
 \begin{theorem}\label{t:links}
 Let $X \sim X(n,p)$ where $n^{-1/k}\ll p$ for a given integer $k\ge 0$.
With high probability, the following properties hold for all simplices $\gs\in X$ of dimension $< k$, with $l:=\dim\gs +1$.
	 \begin{enumerate1}
	 \item
	 $\dim \Lk(\gs) \ge 2k - 2l$.
	\item 
If $i < \lfloor (k-l)/2 \rfloor$, then $\wH^i(\Lk(\gs);\qq)= 0$. 
 \end{enumerate1}
 \end{theorem}
 
 \begin{proof}[Proof of Theorem \ref{t:links}]
 
 The proof of (1) is similar to the proof of statement (1) of Theorem \ref{t:randombetti}.
 
Given a simplex $\sigma \in X(n,p)$ on $l$ vertices, let $N_m$ denote the number of extensions of $\sigma$ to a simplex on $ l+m$ vertices.  This would require a choice of $m$ new vertices out of a possible $n-l$, and then there are $${m+l \choose 2} - {l \choose 2}$$ 
 new edges that must appear.  By linearity of expectation, 
 	\begin{align*}
	 \mathbb{E}[N_m] & = {n - l \choose m} p^{ {l+m \choose 2} - {l \choose 2}}\\
 	& \approx \frac{n^m}{m!} p^{ lm + {m \choose 2} }\\
 	&  = \frac{1}{m!} (np^{l + (m-1) / 2} )^m.\\
 	\end{align*}
Setting $m = 2k - 2l + 1$ gives
$ \mathbb{E}[N_m] = \Theta (np^{mk} )$. 
Since, by assumption, $ p \gg n^{-1/k }$,  $ \mathbb{E}[N_m] \to \infty$.

Janson's inequalities, for example, give that  $$\mathbb{P}[N_m =0] = O(e^{-cn})$$ for some constant $c> 0$.  Since  $\aas$ there are only polynomially many simplices $\sigma$, a union bound gives (1).

The proof of (2) is almost identical to the proof in Theorem \ref{t:randombetti} that $\wH_i(X - \sigma;\qq) = 0$ for every simplex $\sigma$ and for $i < k$.  In particular, there are still only $O(n^{2k})$ simplices $\sigma$ and for each, the probability of failure is $O(n^{-\alpha})$ for every fixed $\alpha > 0$.  So, a union bound shows that the total probability failure is $o(1)$.
 \end{proof} 
 
\paragraph{Some remarks about nonrandom examples.}
Examples of simplicial complexes satisfying the conclusions of Theorems~\ref{t:randombetti} and \ref{t:links} might not spring readily to mind.  Similar properties hold for Cohen-Macaulay complexes, except that for these, the homology is concentrated in the top dimension rather than in the middle.  One can construct examples of complexes satisfying the conclusions Theorems~\ref{t:randombetti} and \ref{t:links} by ``thickening'' certain Cohen-Macaulay complexes.

Let $R$ be a nonzero principal ideal domain (e.g.,  $\zz$ or $\qq$).  A $k$-dimensional complex $Y$ is \emph{Cohen-Macaulay over $R$}  
if for each $\gs\in  Y$, $\wH_i(\Lk(\gs,Y);R)=0$ for $i<k-\dim \gs-1$ and is $R$-torsion-free for $i=k-\dim\gs-1$.  (When $\gs=\emptyset$, $\Lk(\gs,Y)=Y$; so, in this case the condition means that $\wH_i(Y;R)$ is concentrated in degree $k$.) In other words, the link of any $l$-simplex in $Y$ has the same homology as a wedge of $(k-l-1)$-spheres.  A finite simplicial complex $Y$ (of any dimension) has \emph{punctured homology concentrated  in degree $k$}  (with coefficients in $R$) if for each $\gs\in  Y$, $\wH_i(Y-\gs;R)$ is nonzero only in degree $k$ and is $R$-torsion-free in that degree.  
Cohen-Macaulay complexes satisfy a condition similar to the conclusion of Theorem~\ref{t:links} except that the cohomology is concentrated in the top dimension rather than in the middle.  In Theorem~\ref{t:randombetti} we are concerned with the concentration of punctured homology.  Many (but not all) $k$-dimensional Cohen-Macaulay complexes  have the punctured homology concentrated in degree $k$.  For example, any $k$-dimensional spherical building is Cohen-Macaulay and has punctured homology concentrated in degree $k$ (cf. \cite[Thm.~A]{schulz}).  An example of such a spherical building is given by taking the join of any collection of $k+1$ finite sets.

Suppose $Y$ is a $k$-dimensional Cohen-Macaulay complex with concentrated punctured homology.  We can thicken $Y$ to a complex $\wh{Y}$ of dimension $2k$ or $2k+1$ by iterating the procedure of replacing each vertex with a tree (or a forest). This means that we replace the star of a vertex $v$ by the join of the link of $v$ and a forest.  If we do this at each vertex, then $\dim \wh{Y}=2k+1$. By not replacing one vertex of each top-dimensional simplex, we get a $2k$-dimensional $\wh{Y}$.  For example, when $Y$ is a join of finite sets, $\wh{Y}$ is a join of forests.  It is then straightforward to check that such $\wh{Y}$ satisfy the conclusions of Theorems~\ref{t:randombetti} and \ref{t:links}.

\section{Random graph products of groups}\label{s:calc}
As usual, $G\sim G(n,p)$, $X\sim X(n,p)$ and $\cg\sim \cg(G(n,p),\ug)$. 

\subsection{The case where each $\gG_i$ is finite}\label{ss:Gfinite}
In this subsection,  $\gG_i$ is finite of order $q_i+1$ (i.e., $\ug$ has order $\bq+\bone$).  As  we noted in \S \ref{ss:trivial}, the group $\cg_0:=\Ker(\cg\to\ug^{[n]})$ is torsion-free and it acts freely on the universal cover of  the finite complex $\zx(\cone \ug,\ug)$.  Moreover, this universal cover is contractible. So, $\cg_0$ is type F. Since the index of $\cg_0$ in $\cg$ is finite, $\cg$ is type VF.   
 
Let $\car$ denote the region of convergence for $W_{G(n,p)}(\bt)$.

\begin{lemma}\label{l:car2}
If  \,$\sum_{i=1}^\infty (q_i+1)^\minus  < 1$, then $\bone/\bq\in \car$.
\end{lemma}

\begin{proof}
Set $t_i=1/q_i$.  Then $t_i/(1+t_i) =1/(q_i+1)$.  So, if the sum in the lemma is less than $1$, then for all $n\in \nn$,
\[
\sum_{i=1}^{n} \frac{t_i}{(1+t_i)}< 1.
\]
Then, by Lemma~\ref{l:car}, $\bone/\bq \in \car$.
\end{proof}
For example, the conclusion of Lemma~\ref{l:car2} holds if $q_i +1 \ge 2^i$ for all $i\in \nn$.

We begin with some results about the Euler characteristic and $L^2$-Betti numbers of $\cg$.
\begin{proposition}\label{p:randomeuler}

\begin{enumerate1}
\item\label{i:euler}
The rational Euler characteristic of $\cg$ is given by
	\begin{equation*}
	\chi(\cg)=\frac{\chi(B\cg_0)}{(\bone+\bq)_{\n}} 
	=f_{X(n,p)}\left(\frac{-\bq}{\bone+\bq}\right)
	= \frac{\ah_{X(n,p)}(-\bq)}{(\bone +\bq)_{\n}}.
	\end{equation*}
where $\ah_{X(n,p)}$ is defined by \eqref{e:ah}.
\item\label{i:cox}
Let $W_G=\cg(G,\zz/2)$ be the random right-angled Coxeter group.  Then the Poincar\'e series of $BW_G$ with coefficients in $\bF_2$ is given by
\[
	\sum_{i=0}^\infty b_i(BW_G;\ftwo) t^i =f_X\left(\frac{t}{1-t}\right).
\]

\item\label{i:L2}
Suppose $\sum_{i=1}^{\infty}\, (q_i+1)^\minus < 1$. 
Then the $L^2$-Betti numbers $\ltwob_m(\cg)$ are given by
	\[
	\ltwob_m(\cg) =\sum_{\gs\in  X} b_m(\cone X, X-\gs;\qq)\cdot D_\gs (\bq),
	\]
where $D_\gs(\bq)$ is defined by \eqref{e:dgs2}. 
\end{enumerate1}
\end{proposition}

\begin{proof}
Statements \eqref{i:euler}, \eqref{i:cox} and \eqref{i:L2} follow from Proposition~\ref{p:pe}, equation \eqref{e:cox} and Proposition~\ref{p:finL2}, respectively.
\end{proof}

No assumption  on $p$ is made in the above proposition.  The quantities in the equations are all random variables.  The expected values of these quantities can be made completely explicit.  For example, as we saw in \eqref{e:f}, the expected number of $(i-1)$-simplices is given by 
$\expect[f_{i-1}(X)] = { n \choose i} p^{i \choose 2}$.

Recall that a group $\gG$ is a \emph{rational duality group} of \emph{formal dimension} $m$ if it is type $\mathrm{FP}_\qq$ and if $H^*(\gG;\qq \gG)$ is nonzero only in degree $m$.  If this is the case, then, for $D=H^m(\gG;\qq \gG)$ and for any $\qq \gG$-module $M$, $H^i(\gG;M)\cong H_{m-i}(\gG; D\otimes M)$.  

The next result is one of our principal theorems.  It follows from Theorem~\ref{t:randombetti}\,(2) and the results in \S\ref{ss:gpring} and \S\ref{ss:L2betti}.
\begin{theorem}\label{t:Gfin} Fix an integer $k\ge 0$ and suppose
\(
n^{-1/k } \ll p \ll  n^{-1/(k+1) }. 
\)
Then the following properties hold $\eaas$.
\begin{enumerate1}
\item\label{i:middle}
$H^i(\cg;\qq\cg)\neq 0$ if and only if $i=k +1$. Hence, $\cg$ is a rational duality group of formal dimension $k +1$.
\item\label{i:ends}
	\[
	\Ends \cg=
	\begin{cases}
	\infty,		&\text{if $k=0$;}\\
	1, 		&\text{if $k\ge 1$}.
	\end{cases}
	\]
\item\label{i:cd}
The cohomological dimension of $\cg$ over $\qq$  is given by $\cd_\qq \cg =k +1$.  Over $\zz$, the virtual cohomological dimension of $\cg$ 
is either $k +1$ (if $H_{k}(X-\gs)$ is torsion-free for all $\gs\in  X$) or $k +2$ (if $H_{k}(X-\gs)$ has nontrivial torsion for some $\gs\in  X$). 
\item\label{i:L2fin}
Suppose that  $\sum_{i=1}^{\infty}\, (q_i+1)^\minus < 1$. 
Then $\ltwob_m(\cg)$ is nonzero only when  $m=k +1$
\end{enumerate1}
\end{theorem}

\begin{proof}\eqref{i:middle} By Proposition~\ref{p:gpringfin}, $H^i(\cg;\qq\cg)$ is a sum of rational vector spaces of the form $\wH^{i-1}(X-\gs;\qq) \otimes (\hat{A}_\gs\otimes\, \qq)$ where $X\sim X(n,p)$.  So, $H^i(\cg;\qq\cg)\neq 0$  if and only if $\wH^{i-1}(X-\gs;\qq)\neq 0$ for some simplex $\gs$.  By Theorem~\ref{t:randombetti}\,(2), $\wH^{i-1}(X-\gs;\qq)\neq 0$ $\aas$ only for $i=k+1$.

\eqref{i:ends} By Corollary~\ref{c:endsvcd}\,(1), $\Ends(\cg)$ is either $1$ or $\infty$ depending on whether $\wH^0(X-\gs;\qq)$ is zero or not zero.  (The case $\Ends(\cg)=2$ does not occur $\aas$ for $X\sim X(n,p)$.)  By Theorem~\ref{t:randombetti}\,(2) (or the
Erd\H{o}s--R\'enyi Theorem) $\wH^0(X-\gs;\qq)\neq 0$ only when $k=0$.

\eqref{i:cd} As in \S\ref{ss:gpring}, $\cd_\qq(\cg)$ is the largest integer $i$ such that for some simplex $\gs$, $\wH^{i-1}(X-\gs;\qq)\neq 0$.  As before, the largest such $i$ is $k+1$.  As explained in Remark~\ref{remarks}\,(a) below, the second sentence of \eqref{i:cd} follows from Corollary~\ref{c:endsvcd}\,(2).

\eqref{i:L2fin} By Lemma~\ref{l:car2},  $\bone/\bq\in \car$.  By Proposition~\ref{p:randomeuler} \eqref{i:L2}, $\ltwob_m(\cg)\neq 0$ if and only if $b_m(\cone X, X-\gs;\qq)\neq 0$ and by Theorem~\ref{t:randombetti}\,(2), this happens only for $m=k+1$.
\end{proof}

\begin{remarks}\label{remarks}
(a) As in Remark~\ref{r:integer}, the integral homology $H_i(X-\gs)$ vanishes for $i \le (k-1)/2$ or $i > k$. 
Hence, $\wH^i(\cg;\zz \cg)=0$ for $i \le (k+1)/2$ or $i > k +2$.
With regard to statement (\ref{i:cd}) of Theorem~\ref{t:Gfin}, if $H_{k}(X-\gs)$ is torsion-free, then, by the Universal Coefficient Theorem, $H^{k+1}(X-\gs)=0$.  Hence, if $H_{k}(X-\gs)$ is torsion-free for all  $\gs\in X$, then,  by Proposition~\ref{p:gpringfin}, $H^i(\cg;\zz\cg)=0$ for all $i>k +1$.  On the other hand, if $H_{k}(X-\gs)$ has torsion for some simplex $\gs$, then $H^{k+1}(X-\gs)=\Ext(H_{k}(X-\gs),\zz)\neq 0$ and hence, $H^{k +2}(\cg;\zz\cg)\neq 0$.  

One could speculate that $\aas$ $H_i(X-\gs)$ is torsion-free for all $i$ and for all $\gs\in  X$, i.e.,  that $\wH_i=0$ for $i\neq k$ and that $H_i$ is torsion-free for $i=k$ (cf.~Remark~\ref{r:integer}).
If this is true, then $\cg_0$ is an (integral) duality group of formal dimension $k +1$.   In other words, $\cg$ would be a virtual duality group of dimension $k +1$.  

(b)   By statement (\ref{i:cd}) of the theorem, $\cd_\qq \cg_0=k +1$. On the other hand, in Proposition~\ref{p:cg0} we computed the homology of $B\cg_0$ in terms of $H_*(\cone X(I), X(I))$ where $I$ ranges over all subsets of $\n$ which are not vertex sets of simplices.  Hence, (\ref{i:cd}) necessarily entails that $\aas$ $\wH_i(X(I);\qq)=0$, for $i>k$.  The proof of Remark~\ref{r:integer}  given in \cite{clique}  gives a stronger statement with integral coefficients:  $\wH_i(X(I))=0$ for $i>k$ (see \cite[Proof of Thm.~3.6, p.~1667]{clique}).

(c)  It follows from Proposition~\ref{p:randomeuler}\,(\ref{i:euler}) that the sign of $\chi(\cg)$ is  $(-1)^{k +1}$  $\eaas$.  
To see this, first suppose that $\ug$ is the constant sequence, $\gG_k=\gG$, where $\gG$ is a nontrivial finite group.  Then the sign of $\chi(\cg)$ is determined by the fact that the coefficients $f_i$ of the $f$-polynomial are dominated by $f_{k}$.  In fact, for $i\neq k$, $ f_{k}/f_i \to \infty$ as $n\to \infty$.   Moreover, since the order of $\gG$ is an integer $\ge 2$, we have $q\ge 1$.  Hence, the argument of $f_X(\frac{-q}{1+q})=\sum f_i (\frac{-q}{1+q})^{i+1}$ lies between $-1$ and $-1/2$.  Since the absolute value of this  is bounded away from $0$, it follows that the formula for $\chi(\cg)$ is dominated by the term with coefficient $f_{k}$, so $\aas$ its sign is  $(-1)^{k +1}$.  The same argument works when the sequence $\ug$ is not constant.
\end{remarks}

\subsection{The case where each $\gG_i$ is infinite}\label{ss:infinite}
In this subsection, we suppose each $\gG_i$ is infinite.  Once again we begin with some facts about Euler characteristics and $L^2$-Betti numbers.

\begin{proposition}\label{p:randominfinite}\hfil
\begin{enumerate1}
\item\label{i2:euler} 
Suppose each $\gG_i$ is type FL. Let $e_i=e(\gG_i):=\chi(B\gG_i)-1$ be the reduced Euler characteristic of $B\gG_i$, and put $\be=(e_i)_{i\in \nn}$.  Then
	\(
	\chi(\cg)=f_X(\be).
	\)
\item\label{i2:ltwo}  
		\[
			\ltwob_l(\cg)=\sum_{\substack{\gs\in  X\\
			i+m=l}} b_i(\cone \Lk(\gs),\Lk(\gs))\cdot \ltwob_{\gs,m},
		\]
where $ \ltwob_{\gs,m}$ is defined by \eqref{e:bgs}.
\end{enumerate1}
\end{proposition}

\begin{proof}
Statements (\ref{i2:euler}) and (\ref{i2:ltwo}) follow from Propositions~\ref{p:euler1} and \ref{p:infL2}, respectively.
\end{proof}

\begin{remark}
With regard to the formula in Proposition~\ref{p:randominfinite}\,(\ref{i2:ltwo}), $\Lk(\gs)$ can be empty, in which case $\cone \Lk(\gs)$ is a point. 
\end{remark}

When each $\gG_i=\zz$, $\cg\sim \cg(G(n,p),\zz)$ is the random right-angled Artin group $A_G$ associated to the random graph $G\sim G(n,p)$.   Using \eqref{e:ag}, \eqref{e:raag} and Corollary~\ref{c:raag}, we get the following.

\begin{corollary} \textup{(cf.~\cite[Thm.~3.2.4]{cd05}, \cite[Lemma 1]{cof}, \cite{dl}).}
With trivial coefficients the cohomology of $A_G$ is $\aas$ the random exterior face ring $\bigwedge [X]$.  In particular, $b_l(A_G)=f_{l-1}(X)$ and $\chi(A_G)=\chi(\cone X, X)=-e(X)$, where $e$ means reduced Euler characteristic.
\end{corollary}

\begin{theorem}\label{t:infinite} Fix an integer $k\ge0$, suppose
\(
n^{-1/k} \ll p \ll  n^{-1/(k+1) } 
\)
and that $d=\dim X(n,p)$.  
Then the following hold $\eaas$.
\begin{enumerate1}
\item\label{i2:middle}
For $i < k +1$, $H^i(\cg;\qq\cg)=0$ and $H^{k+1}(\cg;\qq\cg)\neq0$.
\item\label{i2:ends}
	\[
	\Ends \cg=
	\begin{cases}
	\infty,		&\text{if $k=0$;}\\
	1, 		&\text{if $k\ge 1$}.
	\end{cases}
	\]
\item\label{i2:cd}
Suppose further that the cohomological dimension of each $\gG_i$ is finite and is equal to $\max\{l\mid H^l(\gG_i;\zz \gG_i)\neq 0\}$.  (This holds, for example, if $\gG_i$ is type FP.)   Then 
	\(
	\cd\cg\le(d+1) \sup\{\cd\gG_i\}
	\).
If $\ug$ is the constant sequence, $\gG_i=\gG$, then $\cd\cg= (d+1)\cd\gG$.  Here, as before, $d=2k$ when $p\le o(n^{-2/(2k+1)})$ or $d=2k+1$ when $p\ge \go(n^{-2/(2k+1)})$. 
\end{enumerate1}
\end{theorem} 

\begin{proof}
Basically,  this follows from the formula in Proposition~\ref{p:gpringinf}.  Here are the details.  Since $\gG_i$ is infinite, $H^0(\gG_i;\zz \gG_i)=0$.  So, for any $l$-simplex $\gs$, by the K\"unneth Formula, $H^i(\ug^{\gs};\qq \ug^{\gs})=0$ for $i<l$; hence, the same vanishing result holds with $\qq \cg$ coefficients.  So, in the formula of Proposition~\ref{p:gpringinf}, for the terms corresponding to $\gs$, the cohomology groups $H^i(\cone\Lk \gs, \Lk \gs)$ are shifted up in degree by at least $l$.  Comparing this with Theorem~\ref{t:links},  we see that, with $\qq \cg$ coefficients, the first degree for which the right hand side of the formula in Proposition~\ref{p:gpringinf} might not vanish is $l +1$ (since $(2k-2l)/2 +l= k$).  So, (\ref{i2:middle}) holds.  Since the number of ends of $\cg$ are detected by $H^1(\cg;\qq \cg)$,  (\ref{i2:middle}) $\implies$ (\ref{i2:ends}).  The formula in Proposition~\ref{p:gpringinf} also implies (\ref{i2:cd}).  To see this, first note that 
	\[
	\cd \ug^{\gs} = \sum _{i\in \gs} \cd \gG_i.
	\]
So, $\cd \ug^{\gs} \le (\dim \gs+1) \sup \{\cd \gG_i\}$.  The nonvanishing terms in the formula of Proposition~\ref{p:gpringinf} which have highest possible degree occur when $\gs$ is a simplex of highest possible dimension $d$, proving (\ref{i2:cd}).
\end{proof}

\begin{corollary} \label{cor:raag}\textup{(cf. \eqref{e:raag}, Corollary~\ref{c:raag}).}
Fix an integer $k \ge 0$ and suppose
\(
n^{-1/k}\ll p \ll n^{-1/(k+1)}. 
\)
Then the following properties  hold $\aas$ for the random right-angled Artin group $A_G$.
\begin{enumerate1}
\item
$\cd A_G = d+1$ where $d=2k$ (when $\go(n^{-2/(2k+1)})\le p$) or $d=2k+1$ (when $p\le o(n^{-2/(2k+1)})$).
\item
$H^i(A_G;\qq A_G)=0$ for $i< k +1$ or $i>d+1$ and $H^{k+1}(A_G;\qq A_G)\neq 0$
\item
$\ltwob_m(A_G)$ is nonzero if and only if  $m=k+1$. 
\end{enumerate1}
\end{corollary}

\begin{proof}
Statements (1) and (2) follow from \eqref{e:raag} and Theorem~\ref{t:links}\,(2). (Statement (1) was first proved in \cite[Thm.\,4]{cof}.)  Statement (3) follows from Corollary~\ref{c:raag} and Theorem~\ref{t:randombetti}\,(2).
\end{proof}

\bibliography{dk_refs}{}
\bibliographystyle{plain}

\obeylines
Michael W. Davis, Department of Mathematics, The Ohio State University, 231 W. 18th Ave., Columbus Ohio 43210  {\tt davis.12@math.osu.edu} 
Matthew Kahle, Department of Mathematics, The Ohio State University, 231 W. 18th Ave., Columbus Ohio 43210  {\tt mkahle@math.osu.edu } 
\end{document}